\documentclass[12pt]{amsart}
 \usepackage{amsfonts,amssymb}

\usepackage{mathrsfs}
\let\mathcal\mathscr

\usepackage[all,ps,cmtip]{xy}

\def\Z{{\bf Z}}

\def\av{abelian variety}

\def\lra{\longrightarrow}
\def\lraa{\lra\hskip-6.05mm \lra}
\def\llra{\hbox to 10mm{\rightarrowfill}}
\def\llraa{\llra\hskip-9.2mm \llra}
\def\lllra{\hbox to 15mm{\rightarrowfill}}
\def\lllraa{\lllra\hskip-14.2mm \lllra}

\def\PA{{\widehat A}}
\def\PB{{\widehat B}}
\def\PK{{\widehat K}}
\def\PV{{\widehat V}}
\def\PX{{\widehat X}}
\def\PE{{\widehat E}}

\def\PT{{\widehat T}}

\def\PF{{\widehat F}}
\def\Pf{{\widehat f}}

\def\phi{{\varphi}}

\def\cF{\mathcal{F}}

\def\cO{\mathcal{O}}
\def\cG{\mathcal{G}}

\def\cH{\mathcal{H}}
\def\cE{\mathcal{E}}

\def\cV{\mathcal{V}}

\def\dra{\dashrightarrow}
\def\tto{\twoheadrightarrow}
\def\isom{\simeq}

\DeclareMathOperator{\isomdra}{\stackrel{{}_{\scriptstyle\sim}}{\dra}}

\def\eps{\varepsilon}

\def\ie{\hbox{i.e.,}}
\def\eg{\hbox{e.g.,}}

\def\gge{globally generated}

\def\vide{\varnothing}

\DeclareMathOperator{\Ker}{Ker}

\DeclareMathOperator{\codim}{codim}
\DeclareMathOperator{\Pic}{Pic}

\DeclareMathOperator{\Irr}{Irr}
\DeclareMathOperator{\Spec}{Spec}

\DeclareMathOperator{\id}{id}
\DeclareMathOperator{\Alb}{Alb}

\DeclareMathOperator{\Card}{Card}

\DeclareMathOperator{\Gal}{Gal}
\DeclareMathOperator{\Td}{Td}

\DeclareMathOperator{\ch}{ch}

\def\subset{\subseteq}

\def\siml{\sim_{\rm lin}}

\newtheorem{lemm}{Lemma}[section]
\newtheorem{theo}[lemm]{Theorem}
\newtheorem{coro}[lemm]{Corollary}
\newtheorem{prop}[lemm]{Proposition}

\newtheorem*{conj*}{Conjecture}

\theoremstyle{definition}

\newtheorem{rema}[lemm]{Remark}
\newtheorem{remas}[lemm]{Remarks}

\newtheorem{exam}[lemm]{Example}

\newtheorem{qu}[lemm]{Question}

\theoremstyle{remark}
\newtheorem*{remark*}{Remark}
\newtheorem*{note*}{Note}

\def\moins{\mathop{\hbox{\vrule height 3pt depth -2pt
width 5pt}\,}}

\begin{document}
\title[Varieties with vanishing Euler characteristic]{Varieties with vanishing holomorphic Euler characteristic}
\author[J. Chen]{Jungkai Alfred Chen}
\address{Taita Institute for Mathematical Sciences\\National Center for Theoretical Sciences, Taipei Office\\
 and Department of Mathematics\\1 Sec. 4, Roosevelt Rd. Taipei 106, Taiwan}
\email{{\tt jkchen@math.ntu.edu.tw}}
\author[O. Debarre]{Olivier Debarre}
\thanks{O. Debarre is part of the  project VSHMOD-2009   ANR-09-
BLAN-0104-01.}
\address{D\'epartement Math\'ematiques et Applications\\UMR CNRS 8553\\\'Ecole Normale Su\-p\'e\-rieu\-re\\45 rue d'Ulm, 75230 Paris cedex 05, France}
\email{{\tt olivier.debarre@ens.fr}}
\author[Z. Jiang]{Zhi Jiang}
\address{Max-Planck-Institut f\"ur Mathematik\\Vivatsgasse 7, 53111 Bonn, Germany}
\email{{\tt flipz@mpim-bonn.mpg.de}}

 \begin{abstract} We study smooth complex projective varieties $X$ of maximal Albanese dimension and of general type  satisfying  $\chi(X, \cO_X)=0$. We prove that the Albanese variety of $X$ has at least three simple factors. Examples were constructed by Ein and Lazarsfeld, and we prove that in dimension 3, these examples are (up to abelian \'etale covers) the only ones. By results of Ueno, another source of examples is provided by varieties $X$   of maximal Albanese dimension and of general type  satisfying  $h^0(X, K_X)=1$. Examples were constructed by Chen and Hacon, and again, we prove that in dimension 3, these examples are (up to abelian \'etale covers) the only ones. We also formulate a conjecture on the general structure of these varieties in all dimensions.
 \end{abstract}

  \maketitle

\section{Introduction}

A smooth complex projective variety $X$ is said to have {\em maximal Albanese dimension} if  its Albanese mapping $X\to \Alb(X)$  is generically  finite (onto its image).

 Green and Lazarsfeld showed  in \cite{gl1} that such a variety satisfies $\chi(X, \omega_X)\ge 0$. Ein and Lazarsfeld later constructed in \cite{el}
 a smooth projective threefold $X$ of maximal Albanese dimension and of general type with
 $\chi(X, \omega_X)= 0$  (see Examples \ref{eel} and \ref{chh}).

We are interested here in describing the structure of  varieties $X$ of maximal Albanese dimension (and of general type) with $\chi(X, \omega_X)=0$.
 This class of varieties is stable by modifications, \'etale covers, and products with any other  variety of maximal Albanese dimension (and of general type). More generally,   if $X$ is a smooth projective variety   of maximal Albanese dimension with a fibration whose general fiber $F$ satisfies  $\chi(F, \omega_F)= 0$, then  $\chi(X, \omega_X)= 0$ (\cite{hp},  Proposition 2.5).

 So we study  smooth  projective varieties $X$ of general type with    $\chi(X,\omega_X)=0 $ and a generically finite morphism $X\to A$ to an \av. In \S\ref{strthe}, we prove a general structure theorem (Theorem \ref{q1}) which implies among other things that {\em $A$ has at least three simple  factors.} Examples where $A$ is the product of any three given non-zero factors can be constructed following Ein and Lazarsfeld, and we speculate that their construction  should (more or less) describe all cases where $A$ has  three simple  factors but, although we prove several results in \S\ref{3sc} in this direction (Propositions \ref{l35}, \ref{surj}, and \ref{l46}) and arrive at the rather rigid picture (\ref{d4}), we are only able to get a complete description   when $X$ has dimension 3: we prove
 that {\em a  smooth projective   threefold $X$ of maximal Albanese dimension and  of general type satisfies $\chi(X,\omega_X)=0$ if and only if it has an abelian \'etale cover which is an Ein-Lazarsfeld threefold}  (Theorem \ref{3cur}).

Another source of examples is provided by varieties $X$ of maximal Albanese dimension and $h^0(X,\omega_X)=1$: it follows from work of Ueno (\cite{ueno}) that they satisfy $\chi(X,\omega_X)=0$. Chen and Hacon constructed  examples  of general type  (see Example  \ref{chh}).
 We gather some properties of these varieties in \S\ref{p11}. However, this class of examples is not stable under \'etale covers and does not lend itself well to our methods of study, {\em except in dimension 3,} where the precise Theorem \ref{3cur} allows us to give {\em a complete description of all     smooth projective threefolds $X$  of maximal Albanese dimension and of general type, such that  $P_1(X)=1$: they are all  modifications of abelian \'etale covers of Chen-Hacon threefolds} (Theorem \ref{th63}).

In \S\ref{s7}, we propose a conjecture on the possible general structure of smooth projective  varieties $X$ of maximal Albanese dimension and of general type satisfying $\chi(X, \omega_X)=0$.
  It seems difficult to give a complete classification, but based on the examples that we know, we conjecture that, after taking modifications and \'etale covers, there should exists  a non-trivial  fibration $X\tto Y$ which is either isotrivial, or whose general fiber $F$ satisfies $\chi(F, \omega_F)= 0$. For the converse, one does have $\chi(X, \omega_X)=0$ in the second case by \cite{hp},  Proposition 2.5, but not necessarily in the first case, of course. Both cases do happen (Example \ref{eell}).

We work over the field of complex numbers.

 \medskip\noindent{\bf Acknowledgements.} The first-named author is partially supported by NCTS and the National
Science Council of Taiwan. This work started during the second-named author's visit to Taipei under the support of the
bilateral Franco-Taiwanese Project Orchid and continued during the first-named
author's visit to Institut Henri Poincar\'e in Paris and the third-named author's stay at the   Max Planck Institut for Mathematics in Bonn. The  authors are grateful for the support they received on these occasions.

\section{Notation and preliminaries}\label{notpre}

For any smooth    projective variety $X$, we set $\PX=\Pic^0(X)$. For $\xi\in\PX$, we will denote by $P_\xi$ an algebraically trivial line bundle on $X$ that represents $\xi$.

Following standard terminology, we will say that a
morphism
$f:X\to A$ to an \av\ $A$
  is {\em minimal} if the induced group morphism $\Pf:\PA\to \PX$  is injective. Equivalently,  $f(X)$ generates $A$ as an algebraic group and $f$ factors through no non-trivial abelian \'etale covers of $A$. The Albanese mapping $a_X$ has this property. Any   $f:X\to A$ factors as $f:X\xrightarrow{f'}A'\to A$, where $A'$ is an \av\ and $f'$ is minimal.

An {\em algebraic fibration} (or simply a fibration) is a surjective morphism between normal projective varieties, with connected fibers.

In the rest of this section, $X$ will be   a  smooth    projective variety, of dimension $n$, with a {\em generically finite}
morphism
$f:X\to A$ to an \av\ $A$.   In particular, $X$ has maximal Albanese dimension.

\subsection{Cohomological loci}\label{21}
For each integer $i$, we define the cohomological loci
\begin{eqnarray*}
V_i(\omega_X,f)&=&\{\xi\in\PA\mid H^i(X,\omega_X\otimes f^*P_\xi)\ne0\}\\
&=&\{\xi\in\PA\mid H^{n-i}(X, f^*P_{-\xi})\ne0\}.
\end{eqnarray*}
If   $Y$ is a smooth   projective variety and   $\eps:Y\tto X$ is birational, we have $R^j\eps_*\omega_Y=0$ for $j>0$ and $\eps_*\omega_Y\isom \omega_X$ (\cite{ko1}, Theorem 2.1), hence  $ \chi(X,\omega_X) = \chi(Y,\omega_Y)$ and $V_i(\omega_X,f)=V_i(\omega_Y,f\circ\eps)$ for all $i$. In particular,  these loci do not change when $X$ is replaced with $Y$.

\subsubsection{}\label{211} Since
$R^jf_*\omega_X=0$ for $j>0$, we have  for all $i$
$$ V_i(\omega_X,f )= V_i(f_*\omega_X ):=\{\xi\in\PA\mid H^i(A,f_*\omega_X\otimes P_\xi)\ne0\}.$$

\subsubsection{}\label{212} Each irreducible component of
$V_i(\omega_X ,f)$
is    an
  abelian subvariety of $\PA$ of codimension  $\ge i $ (\cite{el}, Remark 1.6 and Theorem 1.2) translated   by a torsion point (\cite{sim}).
  
\subsubsection{}\label{213} There is a chain of inclusions (\cite{el}, Lemma 1.8)
\begin{equation*}\label{inc}
\Ker(\Pf)=V_n(\omega_X,f)\subset V_{n-1}(\omega_X,f)\subset \cdots\subset V_0(\omega_X,f)\subset \PA,
\end{equation*}
and $\codim \bigl(V_n(\omega_X,f) \bigr)\ge n$.

\subsubsection{}\label{214} If $V_0(\omega_X,f)$ has a component of codimension $i$, this component is contained in (hence is   an irreducible component of) $V_i(\omega_X,f)$ (\cite{el}, (1.10)), so that we have $i\le n$ and  $f(X)$ is fibered by $i$-dimensional abelian subvarieties of $A$ (\cite{el}, Theorem 3).

\subsubsection{}\label{215}  For $\xi\in \PA$ general,
$\chi(X,\omega_X)=  h^0(X,\omega_X\otimes f^*P_\xi)\ge 0$ (use \ref{212}) and
 \begin{eqnarray*}
\qquad\chi(X,\omega_X)=0 & \Longleftrightarrow&V_0(\omega_X,f)\ne \PA\\
 & \Longleftrightarrow&\hbox{$V_0(\omega_X,f)$ has a component of codimension $i$}\\
 &&\hskip 3cm\hbox{for some $i\in\{1,\dots,n-1\}$}\\
 & \Longrightarrow&\hbox{$V_i(\omega_X,f)$ has a component of codimension $i$}\\
  &&\hskip 3cm\hbox{for some $i\in\{1,\dots,n-1\}$},
  \end{eqnarray*}
where the last implication  is not always an equivalence.

 \subsubsection{}\label{221}The variety $X$ is of general type if and only if $V_0(\omega_X,f)$ generates $\PA$ (\cite{CH}, Theorem 2.3).
  \subsubsection{}\label{222} If $V_0(\omega_X,f)$ is finite, Ein and Lazarsfeld proved that $X$ is birational to an \av\ (\cite{CH}, Theorem 1.3.2). In particular, if $f$ is moreover minimal, it is a birational isomorphism.

\subsection{Composing $f$ with a generically finite morphism}\label{23}
If   $Y$ is   smooth   projective   and $h:Y\tto X$ is  surjective and generically finite,
   the  trace map $h_*\omega_Y\tto\omega_X$  splits the natural inclusion
$\omega_X\to h_*\omega_Y $, hence $H^i(X,\omega_X\otimes f^*P_\xi )
$ injects into $ H^i(X, h_*\omega_Y\otimes f^*P_\xi)$ for all
 $\xi\in\PA$. Thus, using also   \ref{211}, we obtain
$$V_i(\omega_X,f)\subset V_i(h_*\omega_Y,f)=V_i(\omega_Y,f\circ h ). $$
Morover,
\begin{equation}\label{chi}
\chi(Y,\omega_Y)\ge  \chi(X,\omega_X) .
 \end{equation}
 When $h$ is \'etale, we have
 $$
  \chi(Y,\omega_Y) =\deg(h) \chi(X,\omega_X).
 $$
 Finally, when $h$ is obtained from an isogeny $\eta:B\tto A$ as in the   cartesian diagram
  $$
\xymatrix
{Y=X\times_AB\ar@{->>}[d]_h\ar[r]^-{g}\ar@{}[dr]|{\square}&B\ar@{->>}[d]^\eta\\
X\ar[r]_f&A
}
$$
we have  $V_i(\omega_Y,g)=\widehat\eta(V_i(\omega_X,f))$. Combining this with \ref{212}, we see that after making a suitable \'etale base change, we can always make all the components of   $V_i(\omega_X,f)$ pass through 0.

\section{Components of $V_i$ of codimension $i$}\label{strthe}

Let $X$ be a smooth  projective variety with a generically finite morphism $f:X\to A$ to an \av.
If  $\chi(X,\omega_X)=0 $, it follows from \ref{215} that  $V_i(\omega_X,f) $ has a component  of codimension $i$ for some $i\in\{1,\dots,n-1\}$. We prove a structure theorem under this weaker assumption.

\begin{theo}\label{q1}
 Let $X$ be a smooth  projective variety of dimension $n$, let $A$ be an \av, and let $f:X\to A$ be a minimal generically finite morphism. Assume that for some $i\in\{0,\dots,n\}$, the locus $V_i(\omega_X,f) $ has a component $V $ of codimension $i$ in $\PA$. Let $B$ be the abelian variety $ \PV  $,   let $K:=\Ker (A\tto B)^0$,  and assume   $f(X)+K=f(X)$. For a suitable modification $X'$ of an abelian \'etale cover of $X$, the Stein factorization of the morphism $X'\to A\tto B$ induces a surjective morphism $X'\tto Y$ where $Y$   is smooth of dimension $n-i$, of general type, with $\chi(Y,\omega_Y)>0$.
\end{theo}

\begin{rema} \label{remx}
 The condition $f(X)+K=f(X)$ holds:
   \begin{itemize}
 \item  when $f$ is surjective;
 \item  when $V$ is also a component of $V_0(\omega_X,f) $ (\cite{el}, proof of Theorem 3); this applies in particular when $\chi(X,\omega_X)=0 $ (\ref{215}).
 \end{itemize}
 \end{rema}

  \begin{proof}[Proof of Theorem \ref{q1}]
By (\ref{212}) and \S\ref{23}, we may assume, after isogeny,
 that   $A=  B \times K$  and $V=  \PB $.
  Let $p : A \tto B$ be the projection. Considering the Stein factorization of $\pi =p \circ f:X \to  B$, and replacing $X$ by a suitable modification, we may assume that $\pi $ factors as
 $$\pi :X \stackrel{g}{\tto} Y \stackrel{h}{\to}  B,
 $$
where $Y$ is smooth, $h$ is generically finite, and $g$ is surjective with connected   fibers. Since $f(X)+K=f(X)$, the image of $\pi$ has dimension $\dim(X)-\dim(K)$, hence general fibers of $g$ have dimension $\dim(K)=i$.

We then have
${R^ig}_* \omega_X \isom \omega_Y$ (\cite{ko1}, Proposition 7.6).
  Moreover, the sheaves
$R^kg_* \omega_X$ on $Y$
satisfy the generic vanishing theorem (\cite{hp}, Theorem 2.2), hence
$$V_j( R^kg_*\omega_X,h) \ne \PB\quad\hbox{ for all $j>0 $ and all $k$.}$$
For all
$$\xi \in \PB\moins  \bigcup_{j>0,\ k} V_j( R^kg_*\omega_X,h) ,$$
we have
$$H^j(Y,R^kg_*\omega_X\otimes h^*P_\xi )=0\quad\hbox{  for all $j>0$ and all $k$.}$$
Hence, by the Leray spectral sequence, we obtain
$$  h^i(X,\omega_X\otimes f^* P_\xi)=h^i(X,\omega_X  \otimes\pi^*P_\xi)
=h^0(Y,R^ig_*\omega_X \otimes h^*P_\xi )
=h^0(Y,\omega_Y\otimes h^*P_\xi)
$$
and these numbers are non-zero because $ \PB =V\subset  V_i(\omega_X,f)$.
 In particular, $V_0(\omega_Y,h )=\PB$. By \ref{221} and \ref{215},
 $Y$ is of general type and $\chi(Y,\omega_Y)>0$. This completes the proof.
  \end{proof}

  We prove a partial converse to Theorem \ref{q1}: assume that there is a generically finite morphism $f:X\to A$ and a quotient abelian variety $A\tto B$ such that $f(X)+K=f(X)$, where $K:=\Ker (A\tto B)^0$, 
  and denote by $X\dra Y\to B$ a modification of  the 
      Stein factorization of $X\to A\tto B$, where $Y$   is smooth of dimension $n-i$ (we set $i:=\dim(K)$).

\begin{prop} \label{r5}
In this situation, if $Y$ is not birational to an \av, $V_j(\omega_X,f)$ has a component of codimension $j$ for some $j\in\{i,\dots,n-1\}$.
\end{prop}

 \begin{proof}
Replacing $X$ with a modification of an \'etale cover (which is allowed by \S\ref{21} and  \S\ref{23}), we may assume that we have a factorization
 \begin{equation*}\label{sit}
 f:  X\stackrel{(g,k)}{\llraa} Y\times K\stackrel{h\times {\rm Id}_K}{\llra} B\times K,
 \end{equation*}
where $(g,k)$ is surjective and $h:Y\to B$ is  generically finite   of degree $>1$. We obtain, as in the proof of Theorem \ref{q1}, for $\xi$ general in $\PB$,
 \begin{equation}\label{sitt}
 h^i(X,\omega_X\otimes f^*P_\xi)=h^0(Y,R^ig_* \omega_X\otimes h^*P_\xi)=h^0(Y, \omega_Y\otimes h^*P_\xi).
  \end{equation}

{\em If $\chi(Y,\omega_Y)>0$}, we have $V_0(\omega_Y,h)=\PB$, the number on the right-hand-side of (\ref{sitt}) is non-zero for all $\xi$, hence  $V_i(\omega_X,f)$ contains the $i$-codimensional  abelian  subvariety  $\PB$ of $\PA$.

{\em If $\chi(Y,\omega_Y)=0$}, since $Y$ is not birational to an \av,  $V_0(\omega_Y,h)$ has (by \ref{222} and \ref{215}) a component of codimension $l\in\{1,\dots,n-i-1\}$ in $\PB$. Thus,  by Remark \ref{remx},
we can apply   Theorem \ref{q1} to $h:Y\to B$: after taking an \'etale cover and a modification, $h$ factors through  a morphism $Y\to Z\times C$,  where $C$ is an \av\ of dimension $l$, $\chi(Z,\omega_Z)>0$,  and $\dim(Z)=n-i-l$. We are therefore
 reduced to the first case and  we conclude again that $V_{i+l}(\omega_X,f)$ contains an $(i+l)$-codimensional component.
\end{proof}

 \begin{rema}\label{r33}
 Under the hypotheses of Theorem \ref{q1}  and the assumption $A=  B \times K$ made in its proof, we obtain a surjective morphism $k:X\stackrel{f}{\to} B\times K\stackrel{p_2}{\tto}K$ and, from its   Stein factorization, morphisms
  $$ k:X \stackrel{l}{\tto} Z \stackrel{m}{\tto} K,
 $$
where $Z$ is smooth of dimension $i$, $m$ is generically finite, and $l$ has connected (generically $(n-i)$-dimensional) fibers. We have again $R^{n-i}l_*\omega_X\isom  \omega_Z$ and, for $\xi$ general in $\PK$,
$$h^{n-i}(X,\omega_X\otimes f^*P_\xi)=h^0(Z, \omega_Z\otimes m^*P_\xi).$$
Then,
 \begin{itemize}
 \item[a)]  either $V_0(\omega_Z,m)\subsetneq \PK$ and $\chi(Z,\omega_Z)=0$;
  \item[b)]  or   $V_0(\omega_Z,m)= \PK$ and $\chi(Z,\omega_Z)>0$, in which case $\PK$ is contained in (hence is a component of) $V_{n-i}(\omega_X,f)$. There is a surjective generically finite map $X\tto Y\times Z$, hence $\chi(X,\omega_X)\ge \chi(Y,\omega_Y)\chi(Z,\omega_Z)>0$ (this   also follows from Corollary \ref{two}.a) below).
  \end{itemize}
Finally, if $F$ is a general fiber of $l:X\tto Z$, there is a surjective generically finite map $F
 \tto Y$, hence $\chi(F,\omega_F)>0$ (see (\ref{chi})).
 \end{rema}

 We now deduce some consequences of Theorem \ref{q1} on the possible components of $V_0(\omega_X,f) $ and the number of simple factors of the \av\ $A$.

 \begin{coro}\label{two}
 Let $X$ be a smooth  projective variety   with $\chi(X,\omega_X)=0 $ and a   generically finite morphism  $f:X\to A$  to an  \av.
\begin{itemize}
\item[{\rm a)}]
 The locus $V_0(\omega_X,f) $ does not have complementary components.\footnote{By that, we mean components such that the sum morphism induces an isogeny from their product onto $\PA$.}
 \item[{\rm b)}] If $X$ is in addition of general type, $A$ has at least three simple factors.
 \end{itemize}
 \end{coro}

  \begin{proof}
  If $V_0(\omega_X,f) $ has complementary components $V_1,\dots,V_r$, with duals $B_1,\dots,B_r$, the   image   $f(X)$ is stable by translation by $\prod_{j\ne i}B_j$ for each $i$ (Remark \ref{remx}), hence $f$ is surjective if $r\ge 2$. We obtain from  
  Theorem \ref{q1}, after passing to an \'etale cover and a modification of $X$, a generically finite   surjective map $X\tto Y_1\times\dots\times Y_r$, with $\chi(Y_i,\omega_{Y_i})>0$ for all $i$. Since $\chi(X,\omega_X)\ge \prod_i \chi(Y_i,\omega_{Y_i}) $ (by (\ref{chi})), this is absurd. This proves a).
  Item b) then follows from  \ref{221}.
   \end{proof}

\begin{prop}\label{dim1}
Let $X$ be a smooth  projective variety   of general type,  with $\chi(X,\omega_X)=0 $ and a   generically finite morphism  $f:X\to A$  to an  \av.  Then $V_0(\omega_X, f)$ has no 1-dimensional   components.
\end{prop}

\begin{proof}Assume $V_0(\omega_X, f)$ has a  one-dimensional component and write it as $\tau_1+ \PB_1$ for some torsion point $\tau_1\in \PA$ and some quotient elliptic curve $A\tto B_1$.
 By \cite{jia}, Proposition 1.7,  and  since $V_0(\omega_X, f)$ generates $\PA$ (\ref{221}), $\PB_1$ cannot be maximal for inclusion: more precisely, there must exist
   two maximal components (in the sense of \cite{jia}, Definition 1.6) $\tau_2+\PB_2$ and $\tau_3+\PB_3$ of $V_0(\omega_X, f)$ with $\PB_1\subsetneq \PB_j\subsetneq \PA$ for each $j\in\{2,3\}$ and $ \PB_2\ne \PB_3$. There are corresponding factorizations $A\tto B_j\tto B_1$.

As in the proof of Theorem \ref{q1}, after passing to an \'etale cover and a modification of $X$, we may assume $\tau_1=\tau_2=\tau_3=0$ and that we have
\begin{itemize}
\item Stein factorizations
$$X\stackrel{g_i}{\lraa} Y_i\xrightarrow{h_i} B_i,$$ where $h_1$, $h_2$, and $h_3$ are generically finite, $Y_1$, $Y_2$, and $Y_3$ are smooth,    $Y_1$ is a curve of genus $\ge 2$, and $\chi(Y_2,\omega_{Y_2})$ and $\chi(Y_3,\omega_{Y_3})$ are both positive (Theorem \ref{q1}),
\item   a commutative diagram
$$
\xymatrix{
&Y_2\ar@{->>}[dr]^{h_{21}}\ar[rr]^{h_2}&&B_2\\
 X\ar@{->>}[dr]_{g_3}\ar@{->>}[ur]^{g_2} \ar@{->>}[rr]^{g_1}&&Y_1 \\
&Y_3\ar@{->>}[ur]_{h_{31}}\ar[rr]_{h_3}&&B_3  .}
$$
\end{itemize}
We may further assume that  the induced morphism $  X\to Y_2\times_{Y_1}  Y_3$ factors as:
$$
\xymatrix{
&&&Y_2\ar@{->>}[dr]^{h_{21}}\ar[rr]^(.6){h_2}&&B_2\\
 X\ar[r]\ar@{->>}@/_2pc/[rrrr]_(.4){g_1}&Y\ar@{->>}[r]_-{\eps}\ar@{->>}@/^1.5pc/[rrr]^(.35){q}&Y_2\times_{Y_1} Y_3\ar@{->>}[dr]|(.53){\vbox to 2mm{\hglue 1mm}}\ar@{->>}[ur]|(.43){\vbox to 2mm{\hglue 1mm}}  \ar@{->>}[rr]  &&Y_1 ,\\
&&&Y_3\ar@{->>}[ur]_{h_{31}}\ar[rr]_(.6){h_3}&&B_3   }
$$
where $\eps$ is a resolution of singularities.

Now take     $\xi_2\in  \PB_2$ and $\xi_3\in  \PB_3$. By \cite{mor}, Lemma 4.10.(ii),\footnote{This is stated in \cite{mor} for $\xi_2=\xi_3=0$, but the same proof works in general.}  there is an inclusion
$$q_*(\omega_{Y /Y_1}\otimes \eps^*(h_2^*P_{\xi_2}\otimes h_3^*P_{\xi_3}))\subset h_{21*}(\omega_{Y_2/Y_1}\otimes h_2^*P_{\xi_2})\otimes h_{31*}(\omega_{Y_3/Y_1}\otimes h_3^*P_{\xi_3}),$$
of locally free sheaves of the same rank on the curve $Y_1$. Moreover, we saw during the proof of Theorem \ref{q1} that for $j\in\{2,3\}$, we have
$$0\ne h^0(Y_j, \omega_{Y_j}\otimes h_j^*P_{\xi_j})=h^0(Y_1, h_{j1*}(\omega_{Y_j}\otimes h_j^*P_{\xi_j})).$$
 It follows that the sheaf $h_{j1*}(\omega_{Y_j}\otimes h_j^*P_{\xi_j})$ is non-zero, hence so is the sheaf $h_{j1*}(\omega_{Y_j/Y_1}\otimes h_j^*P_{\xi_j})$. All in all, we have obtained that the locally free sheaf
$q_*(\omega_{Y /Y_1}\otimes \eps^*(h_2^*P_{\xi_2}\otimes h_3^*P_{\xi_3}))$ is non-zero.

Assume now that      $\xi_2 $ and $\xi_3 $ are torsion. By      \cite{V},  Corollary 3.6,\footnote{This is stated there for $\xi_2=\xi_3=0$, but the general case follows by the \'etale covering trick.} this vector bundle   is nef, hence has non-negative degree. Since $Y_1$ is a curve of genus $\ge 2$, the Riemann-Roch theorem then implies
$$0\ne h^0(Y_1, q_*(\omega_{Y }\otimes \eps^*(h_2^*P_{\xi_2}\otimes h_3^*P_{\xi_3})))=h^0(Y, \omega_Y\otimes \eps^*(h_2^*P_{\xi_2}\otimes h_3^*P_{\xi_3})).$$
Finally, note that both $X$ and $Y$ have maximal Albanese dimensions. This implies that $\omega_{X/Y}$ is effective, hence $h^0(X, \omega_X\otimes  f^*( P_{\xi_2}\otimes  P_{\xi_3}))$ is also non-zero.
 It follows  that $\xi_2+\xi_3 $ is in $V_0(\omega_X,f)$, which therefore contains  $  \PB_2+\PB_3$.
 This contradicts the fact that $ \PB_2$ is   maximal.
\end{proof}

\section{Case when $A$ has three simple factors}\label{3sc}

 Ein and Lazarsfeld constructed an example of a smooth  projective threefold $X$  of maximal Albanese dimension and of general type with $\chi(X,\omega_X)=0 $, whose Albanese variety is the product of three elliptic curves.  After presenting their construction (and a variant due to Chen and Hacon), we prove some general results when $\Alb(X)$ has three simple factors. In the next section, we will show that the Ein-Lazarsfeld example is essentially the only one in dimension 3 (Theorem \ref{3cur}).

 \begin{exam}[\cite{el}, Example 1.13]\label{eel}  Let $E_1$, $E_2$, and $E_3$ be   elliptic
curves and let $\rho_j : C_j \tto E_j$ be   double coverings, where $C_j$ is a smooth
curve of genus $\ge 2$ and $\rho_{j*}\omega_{C_j}\isom\cO_{E_j}\oplus  \delta_j$. Denote by $\iota_j$ the corresponding involution of $C_j$. Let
$A = E_1 \times E_2 \times E_3$, and   consider the quotient $Z$ of $C_1 \times C_2 \times C_3$ by the involution
$\iota_1
\times \iota_2 \times \iota_3$ and the tower of
Galois covers:
$$C_1 \times C_2 \times C_3 \stackrel{g}{\lraa}Z \stackrel{f}{\lraa}A$$ of degrees $2$ and $4$ respectively. Observe that $Z$   has   rational singularities and  is minimal of general type. Let $\eps: X\tto Z$ be any desingularization.
The Albanese map of $X$ is $a_X=f\circ \eps$ and
$$a_{X*}\omega_X\isom \cO_A\oplus (L_1\otimes L_2)\oplus (L_3\otimes L_1)\oplus (L_2\otimes L_3),
$$
where $L_j$ is the inverse image of $\delta_j$ by the  projection $A\tto E_j$, hence
\begin{equation}\label{aX}
  V_0(\omega_X, a_X)=V_1(\omega_X, a_X)
 =(\widehat E_1\times \widehat E_2\times\{0\})\cup(\widehat E_1\times\{0\}\times \widehat E_3)\cup(\{0\}\times\widehat E_2\times \widehat E_3),
\end{equation}
whereas $V_2(\omega_X, a_X)=V_3(\omega_X, a_X)=\{0\} $.

 This provides three-dimensional examples. Obviously,   the same construction  works starting from double coverings $\rho_j : X_j \tto A_j$ of abelian varieties  with smooth ample branch loci and provides examples in all dimensions $\ge 3$. One can also extend it to any {\em odd} number $2r+1$ of factors and get examples where the Albanese mapping is birationally a $(\Z/2\Z)^{2r}$-covering.
\end{exam}

  \begin{exam}[\cite{ch2}, \S4, Example]\label{chh} A variant of the construction above was given by Chen and Hacon. Keeping the same notation, choose points $\xi_j\in\PE_j$ of order 2 and consider the induced double \'etale covers $C'_j\tto C_j$, with associated involution $\sigma_j$, and $E'_j\tto E_j$. The involution $\iota_j$ on $C_j$ pulls back to an involution $\iota'_j$ on $C'_j$ (with quotient $E'_j$). Let $Z'$ be the quotient of $C'_1\times C'_2\times C'_3$ by the  group of automorphisms generated by $\id_1\times \sigma_2\times\iota'_3$, $\iota'_1\times\id_2\times \sigma_3$, $\sigma_1\times\iota'_2\times\id_3$, and $\sigma_1\times\sigma_2\times\sigma_3$, and let $\eps': X'\tto Z'$ be a desingularization. There is
 a morphism $f':X'\tto A$ of degree 4,
  the Albanese map of $X'$ is $a_{X'}=f'\circ \eps'$, and
  \begin{equation}\label{e5}
a_{X'*}\omega_{\widetilde X'}\isom \cO_A\oplus (L_1\otimes L^\xi_2\otimes P_{\xi_3})\oplus (L^\xi_1 \otimes P_{\xi_2}\otimes L_3 )\oplus (P_{\xi_1}\otimes L_2\otimes L^\xi_3),
\end{equation}
where
$L^\xi_j=L_j\otimes P_{\xi_j}$. In particular, $P_1(X')=1$, and
 $$V_0(\omega_{X'}, a_{X'})=V_1(\omega_{X'}, a_{X'})=\{0\}\cup (\PE_1\times \PE_2\times\{\xi_3\})\cup(\PE_1\times\{\xi_2\}\times \PE_3)\cup(\{\xi_1\}\times\PE_2\times \PE_3).
$$
Of course, the \'etale cover $E'_1\times E'_2\times E'_3\tto E_1\times E_2\times E_3$ pulls back to an \'etale cover $X''\to X'$, where $X''$ is an Ein-Lazarsfeld threefold.
 
 Again, this construction still works starting from double coverings   of abelian varieties  with smooth ample branch loci, providing examples in all dimensions $\ge 3$, and for  any {\em odd} number $2r+1$ of factors, providing  examples where the Albanese mapping is birationally a $(\Z/2\Z)^{2r}$-covering.
\end{exam}
\medskip


\begin{prop}\label{l35}
 Let $X$ be a smooth  projective variety of general type  with $\chi(X,\omega_X)=0 $ and a   generically finite morphism  $f:X\to A$  to an  \av\ $A$ with exactly  three simple factors  $A_1$, $ A_2$,   $ A_3$.
 \begin{enumerate}
\item[{\rm a)}]
The   map $f$ is surjective.
 \item[{\rm b)}]  After passing to an abelian \'etale cover,  we may assume $A=A_1\times A_2\times A_3$ and that $\PA_1\times \PA_2\times \{0\}$,
 $\PA_1\times
\{0\}\times \PA_3$, and $ \{0\}\times \PA_2\times \PA_3$ are irreducible
components of $V_0(\omega_X,f)$.
 \end{enumerate}
 \end{prop}

 \begin{proof}
 We begin with the proof of item b). As in the proof of Theorem \ref{q1}, we  may assume, after passing to an abelian \'etale cover,  that $A$ is $A_1\times A_2\times A_3$ and, by Corollary \ref{two}.a) and \ref{221},   that $\PA_1\times \PA_2\times \{0\}$ and
 $\PA_1\times
\{0\}\times \PA_3$ are irreducible
components of $V_0(\omega_X,f)$.  

Assume that the projection $V_0(\omega_X,f)\to \PA_2\times \PA_3$ is not surjective. For $\xi_2$ and $\xi_3$ general  torsion points   in $ \PA_2$ and $ \PA_3$ respectively, we then have
 $$  (\xi_2+\xi_3+\PA_1)\cap V_0(\omega_X,f)=\vide. $$
   Consider the morphism $f_1=p_1\circ f:
X\to A_1$ and the sheaf  $\cE=f_{1*}(\omega_X\otimes f_2^*P_{\xi_2}\otimes f_3^*P_{\xi_3})$ on $A_1$. By \cite{hp}, Theorem 2.2, the cohomological loci of $\cE$ satisfy the chain of inclusions \ref{213}. On the other hand, for all $\xi\in \PA_1$, we have
$H^0(A_1, \cE\otimes P_\xi)=0 $, hence $V_0(\cE)=\vide$.
It follows that for all $\xi\in \PA_1$ and all $i\ge 0$, we have
$H^i(A_1, \cE\otimes P_\xi)=0 $. This implies $\cE=0$ by Fourier-Mukai duality (\cite{muk2}). But the rank of $\cE$ is at least $ \chi(F_1,\omega_{F_1})$, where $F_1$ is a component of a general fiber of $f_1$ (\cite{hp}, Corollary 2.3) and this is impossible: $F_1$ is of general type and generically finite over $A_2\times A_3$, hence $ \chi(F_1,\omega_{F_1})>0$ (Corollary \ref{two}.b)).

The projection $V_0(\omega_X,f)\tto \PA_2\times \PA_3$ is therefore surjective. Since $A_1$ is simple, this implies that, after passing to a (split) \'etale cover of $A$, there are morphisms $u_2: \PA_2\to \PA_1$ and $u_3: \PA_3\to \PA_1$ such that
$$\{u_2(\xi_2)+u_3(\xi_3)+\xi_2+\xi_3\mid \xi_2\in \PA_2,\ \xi_3\in \PA_3\}$$
is a component of $V_0(\omega_X,f)$. Composing this cover with the automorphism $(a_1,a_2,a_3)\mapsto (a_1,a_2-\widehat u_2(a_1),a_3-\widehat u_3(a_1))$ of $A$, we obtain b).

Item a) then follows from the fact that $f(X)$ is stable by translation by each $A_i$ (Remark \ref{remx}) hence is equal to $A$. 
 \end{proof}

\begin{rema}\label{r44}
As shown by considering the product with a curve of genus $\ge 2$ of any  variety $X$ of general type and maximal Albanese dimension   with $\chi(X,\omega_X)=0 $, the conclusion of Proposition \ref{l35}.a) does not   hold in general as soon as $A$ has at least four simple factors.
\end{rema}

 \begin{prop}\label{surj}
  Let $X$ be a smooth  projective variety of general type  of dimension $n$ with $\chi(X,\omega_X)=0 $ and a   generically finite   morphism $ X\to A$  to an  \av\ $A$ with exactly three simple factors. We have:
  \begin{enumerate}
  \item[{\rm a)}]  $q(X)=n$;
   \item[{\rm b)}] the general fiber $F$ of any non-constant fibration $X\tto Y$ satisfies $\chi(F,\omega_F)>0 $;
     \item[{\rm c)}] any morphism from $X$ to a curve of genus $\ge 2$ is constant;
      \item[{\rm d)}] $V_{n-1}(\omega_X,a_X)=\{0\}$.
     \end{enumerate}
   \end{prop}

 \begin{proof}
 Let us prove b) first. The fiber $F$ is of general type and is generically finite but not surjective over $A$, hence $\chi(F,\omega_F)>0 $ by Proposition \ref{l35}.a). The other items then follow from the   lemma below.\end{proof}

\begin{lemm}\label{le46}
Let $X$ be a smooth projective variety of dimension $n$, of maximal Albanese dimension, of general type, with $\chi(X, \omega_X)=0$. Assume that  the general fiber $F$ of any non-constant fibration $X\tto Y$ satisfies $\chi(F,\omega_F)>0 $. Then,
  \begin{enumerate}
  \item[{\rm a)}]  the Albanese mapping $a_X $ is surjective ($q(X)=n$);
   \item[{\rm b)}] any morphism from $X$ to a curve of genus $\ge 2$ is constant;    
    \item[{\rm c)}]  $V_{n-1}(\omega_X,a_X)=\{0\}$.
     \end{enumerate}
 \end{lemm}

\begin{proof}Item a) follows from   \cite{ch2}, Theorem 4.2,   and item b) from \cite{hp}, Theorem 2.4. Let us prove c).

  If $V_{n-1}(\omega_X,a_X)\moins\{0\}  $ is non-empty, it contains a   torsion point  by \ref{212}, which defines a connected \'{e}tale cover $\pi: \widetilde{X}\tto X$ such that $q(\widetilde{X})>q(X)=n$. By \cite{ch2}, Theorem 4.2, again,  there exists a non-constant fibration $ \widetilde{X}\tto Y$ with   general fiber $F$ of maximal Albanese dimension, of general type, with $\chi(F, \omega_F)=0$, such that $a_{\widetilde{X}}(F)$ is a translate of a fixed abelian subvariety $\widetilde  K$ of $\Alb(\widetilde{X})$ and $\dim (\widetilde  K)=\dim (F)<\dim (X)$. We consider the image $K$ of $\widetilde K$ by the  induced map $\Alb(\pi):\Alb(\widetilde X)\tto \Alb(X)$ and   the   commutative diagram:
$$
\xymatrix@C=40pt@M=6pt{
F\ar@{->>}[r]^{\pi\vert_F}\ar@{_{(}->}[d]&\pi(F)\ar@{_{(}->}[d]\ar@{->>}[r]^-{ a_{\widetilde{X}}\vert_{\pi(F)}}&  K +x_F \ar@{_{(}->}[d]\\
\widetilde{X}\ar@{->>}[d]_{f}\ar@{->>}[r]^{\pi}& X\ar@{->>}[r]^-{a_X}\ar@{->>}[dr]_-h& \Alb(X)\ar@{->>}[d]\\
Y&& \Alb(X)/ K .}
$$
The map $a_{\widetilde{X}}\vert_{\pi(F)}$ is generically finite, hence  $\dim (K)=\dim\pi(F)=\dim (F)$, and 
$$\dim (\Alb(X)/ K ) =n-\dim(K)=n-\dim(F).$$
It follows that $\pi(F)$ is a general fiber of the Stein factorization of $h$. Since 
 $\chi(\pi(F), \omega_{\pi(F)})=0$, this contradicts our hypothesis on $X$.
\end{proof}

 Assume now that we are in the situation of Proposition \ref{l35}.b) and consider, as in Remark \ref{r33}, the 
  maps $f_i:=p_i\circ f:X\tto A_i$. Since $ \chi(X,\omega_X) =0$, we are in  case a) of that remark,  hence, with the notation therefrom, $V_0(\omega_Z,m)$ must be finite, because $A_i$ is simple.
 If $f$ is minimal, so is $f_i$ and it    follows from
  \ref{222}  that $Z$ is birational to $A_i$, hence $f_i $ is a fibration. A general fiber $F_i$ satisfies $ \chi(F_i,\omega_{F_i})>0$ by Proposition \ref{surj}.b).

  \begin{prop}\label{l46} Let $X$ be a smooth  projective variety of general type  with $\chi(X,\omega_X)=0 $ and a minimal  generically finite morphism  $f:X\to A$  to an  \av\ $A$ product    of  three simple factors  $A_1$, $ A_2$, and $ A_3$.
  
  Let $\{i,j,k\}=\{1,2,3\}$.
For $\xi_j$ and $\xi_k$ general  torsion points   in $ \PA_j$ and $ \PA_k$ respectively,   the sheaf  $f_{i*}(\omega_X\otimes f_j^*P_{\xi_j}\otimes f_k^*P_{\xi_k})$ on $A_i$ is locally free, homogeneous, of positive rank $ \chi(F_i,\omega_{F_i})$.
 \end{prop}

 \begin{proof}We follow the proof of \cite{ch2}, Corollary 2.3. As in the proof of Proposition \ref{l35}, since $\xi_j$ and $\xi_k$ are torsion,  the cohomological loci of $\cE:=f_{i*}(\omega_X\otimes f_j^*P_{\xi_j}\otimes f_k^*P_{\xi_k})$ satisfy the chain of inclusions \ref{213}. On the other hand, since $\xi_j$ and $\xi_k$ are
 general and $A_i$ is simple, the intersection
 $$  (\xi_j+\xi_k+\PA_i)\cap V_0(\omega_X,f)  $$
 is finite by \ref{212}, hence so is  $V_0(\cE)$. It follows that all $V_l(\cE)$ are finite, hence $ \cE $ is locally free and homogeneous (\cite{muk2}, Example 3.2). Its rank is
$h^0(F_i, \omega_{F_i}\otimes f_j^*P_{\xi_j}\otimes f_k^*P_{\xi_k})=\chi(F_i,\omega_{F_i})
 $.
 \end{proof}

 \begin{rema}
 Recall that a homogeneous vector bundle is a direct sum of twists of unipotent vector bundles (successive extensions of trivial line bundles) by algebraically trivial line bundles which, in our case, are torsion by Simpson's theorem (or rather its extension \cite{hp}, Theorem 2.2.b)).

   {\em When $A_i$ is an elliptic curve,} the   sheaf   of Proposition \ref{l46} is actually a direct sum of torsion   line bundles (this is
 explained at the bottom of page 362 of \cite{ko3} when $\xi_j=\xi_k=0$ and holds in general by the \'etale covering trick).
 \end{rema}

 Finally, from the proof of Theorem \ref{q1}, we have (after replacing $X$ with a suitable modification), for each $\{i,j,k\}=\{1,2,3\}$, Stein factorizations
 $$p_{jk}\circ f:X\stackrel{g_i}{\llraa} S_i\stackrel{(h_{ij}, h_{ik})}{\lllraa} A_j\times A_k,$$
 where  $S_i$
is smooth of general type with $\chi(S_i,\omega_{S_i})>0$ (this follows also from Corollary \ref{two}.b)). Since $f_j$ has connected fibers, so does $h_{ij}:S_i\tto A_j$.

All in all, we have for each $\{i,j,k\}=\{1,2,3\}$ a commutative diagram:
\begin{equation}\label{d4}
\xymatrix@C=40pt{
&&X\ar@{->>}[dl]^{g_i}\ar@{->>}@/_1.5pc/[ddll]_{f_j}\ar@{->>}[dr]_{g_j}\ar@{->>}@/^1.5pc/[ddrr]^{f_i}\ar@{->>}[dd]_{f_k} \\
&S_i\ar@{->>}[dr]_{h_{ik}}\ar@{->>}[dl]^{h_{ij}}&&S_j\ar@{->>}[dl]^{h_{jk}} \ar@{->>}[dr]_{h_{ji}}\\
A_j&&A_k &&
A_i, }
\end{equation}
where all the morphisms are fibrations.

\begin{qu}
Are the $h_{ij}$   isotrivial? Are the $f_i$   isotrivial? Is $X$ rationally dominated by a product  $X_1\times X_2\times X_3$, where $X_i$ dominates and is  generically finite over $A_i$? We are inclined to think that the answers to all these questions should be affirmative, but we were only able to go further in the case where the $A_i$ are all elliptic curves.
\end{qu}

\section{The 3-dimensional case}\label{s3}

We now come to our main result, which completely describes all   smooth projective   threefolds $X$ of maximal Albanese dimension and  of general type, with $\chi(X,\omega_X)=0$. 

\begin{theo}\label{3cur}
Let   $X$ be a  smooth projective   threefold of maximal Albanese dimension and  of general type, with $\chi(X,\omega_X)=0$. 

There exist elliptic curves $E_1$, $ E_2$, and $ E_3$, double coverings $C_j\tto E_j$ with associated involutions $\iota_j$, and a commutative diagram
$$
\xymatrix@C=15pt
 {&(C_1\times C_2\times C_3)/\iota_1\times \iota_2\times\iota_3 \ar@{->>}[dr]\\
\widetilde X\ar@{->>}[rr]^{a_{\widetilde X}}\ar@{}[drr]|\square\ar@{->>}[ur]^-\eps\ar@{->>}[d]&& E_1\times E_2\times E_3\ar@{->>}[d]^\eta\\
X\ar@{->>}[rr]_{a_X} &&\Alb(X), }
$$
where $\eta$ is an isogeny and $\eps$ a desingularization.
\end{theo}

In other words, up to abelian  \'etale covers, the Ein-Lazarsfeld examples (Example \ref{eel}) are the only ones (in dimension 3)! Note also that $a_X$ is not   finite, but that it is finite on the  canonical model of $X$.

\begin{coro}\label{c52}
Under the hypotheses of the theorem, the Albanese mapping of $X$ is birationally a $(\Z/2\Z)^2$-covering.
\end{coro}

\begin{proof}With the notation of Example \ref{eel}, we set $  A:= E_1\times E_2\times E_3$ and
 $\widetilde X_i:=\Spec (\cO_A\oplus L^\vee_i)$, so that $f$ factors through the double coverings $\widetilde f_i:=\widetilde X_i\tto   A$. Since the action of $ \Ker(\eta)$ on $  A$ by translations lifts to $\widetilde X$, it leaves 
$$a_{\widetilde  X*}\cO_{\widetilde  X}\isom \cO_A\oplus (L^\vee_1\otimes L^\vee_2)\oplus (L^\vee_3\otimes L^\vee_1)\oplus (L^\vee_2\otimes L^\vee_3) 
$$
invariant, hence also each $L^\vee_i$. It follows that this action lifts to each $\widetilde X_i$, hence $\widetilde f_i$ descends to a double covering $X_i\tto \Alb(X)$ through which $a_X$ factors.
\end{proof}

\begin{proof}[Proof of the theorem]
 By Proposition \ref{surj}.a), $a_X$ is surjective and $q(X)=3$ (see also \cite{ch2}, Corollary 4.3).

Moreover, by Corollary \ref{two}.b),  $\Alb(X)$ is isogeneous to the product of three elliptic curves and, after passing to   \'etale covers, we may assume that   $a_X$ can be written as
 $$a_X: X\stackrel{(f_1,f_2,f_3)}{\lllraa} E_1\times E_2\times E_3,$$
 where each
$f_i: X\tto E_i$  is a fibration, 
and that
$\PE_1\times \PE_2\times \{0\}$,
 $\PE_1\times
\{0\}\times \PE_3$, and $ \{0\}\times \PE_2\times \PE_3$ are irreducible
components of $V_0(\omega_X,a_X)$  (Proposition \ref{l35}.b)).

Let $\{i,j,k\}=\{1,2,3\}$. As in \S\ref{3sc}, we have  a commutative diagram (\ref{d4}), where each $S_i$ is a  smooth
minimal surface  of general type and $A_i=E_i$. The proof of the theorem is very long, so we will divide it in several steps. The general scheme of proof goes as follows:
\begin{itemize}
\item In the diagram (\ref{d4}), the fibrations $h_{ij}:S_i\tto E_j$  are all isotrivial (Step 1); we let $C_{ij}$ be a (constant) general fiber.
\item There exist   finite groups $G_i$ acting on $C_{ij}$   such that $C_{ij}/G_i\isom E_j$ and $C_{ik}/G_i\isom E_k$, the surface $S_i$ is birational to  $(C_{ij}\times C_{ik})/G_i$, and $h_{ij}$ and $h_{ik}$ are the two projections (Step 2).
\item At this point, it is quite easy to show that $X$ dominates a threefold $Y$ which is dominated by a product of 3 curves (Step 3). 
\item Taking an \'etale cover of $X$, we may assume   that all the irreducible components of $V_0(\omega_X, a_X)$ pass through $0$. We then show  (Step 4) that $V_0(\omega_X, a_X)$  has the same form as the corresponding locus of an Ein-Lazarsfeld threefold (see (\ref{aX})), from which we deduce that $X$ is birationally isomorphic to $Y$ (Step 5) hence is also dominated by a product of 3 curves.
\item Using the fact that $V_0(\omega_X, a_X)$ has no ``extra'' components, we finish the proof by showing that the groups $G_i$ all have order 2 (Step 6).
\end{itemize}

\noindent{\bf Step 1.}
{\em The fibrations $h_{ij}:S_i\tto E_j$  are all isotrivial.}

We will denote by $C_{ij}$ a general (constant) fiber of $h_{ij}$.

\begin{proof}
By the semi-stable reduction theorem (\cite{KKMS}, Chapter II),
there exist a finite cover $h: C\tto E_j$,  where $C$ is a
smooth curve and    commutative
diagrams (for each $\alpha\in\{i,k\}$)
$$
\xymatrix{S'_\alpha\ar@{->>}[dr]_{h_\alpha}\ar@{->>}[r]^-{\eps_\alpha}
& C\times_{E_j}S_\alpha\ar@{->>}[d]\ar@{->>}[r] &S_\alpha\ar@{->>}[d]^{h_{\alpha  j}}\\
& C\ar@{->>}[r]^{h} &E_j, }
$$
where    $\eps_\alpha$ is a   modification and the
fibers of $h_\alpha$ are all  reduced connected curves, with non-singular components
crossing transversally.

We also make a   modification $\tau : X'\to
C\times_{E_j}X$ such that there exists a commutative diagram of
morphisms between smooth   varieties:
$$
\xymatrix@C=40pt{
&&X' \ar@{->>}[dl]^{g'_i}\ar@{->>}@/_1.5pc/[ddll]_{f'_k}\ar@{->>}[dr]_{g'_k}\ar@{->>}[dd]_(0.6){f'}\ar@{->>}@/^1.5pc/[ddrr]^{f'_i} \\
&S'_i\ar@{->>}[dl]^{h'_{ik}} \ar@{->>}[dr]_{h_i}&&S'_k\ar@{->>}[dl]^{h_k}\ar@{->>}[dr]_{h'_{ki}}\\
 E_k&&C&&E_i.}
$$

Let $\xi_\alpha\in \PE_\alpha$. By  \cite{V}, Lemma 3.1,\footnote{This is proved there for $\xi_i=\xi_k=0$, but the same proof works in general.}   we have an inclusion
\begin{equation}\label{4}
f'_*(\omega_{X'/C}\otimes {f'_i}^*P_{\xi_i}\otimes
{f'_k}^*P_{\xi_k})\subset h^*f_{j*}(\omega_X\otimes f_i^*P_{\xi_i}\otimes
f_k^*P_{\xi_k})
\end{equation}
of locally free sheaves on $C$.
Since $h_\alpha$ is flat with irreducible general fibers, $S'_i\times_C S'_k$ is irreducible, and
we have a surjective morphism
$$g'_{ik}: X'\stackrel{(g'_i, g'_k)}{\lllraa} S'_i\times_C S'_k.$$
Moreover, $h_i$ and $h_k$ are semistable, hence by
\cite{AK}, Proposition 6.4, $S'_i\times_CS'_k$ has only rational Gorenstein
singularities. After further modification of $X'$, we may assume that $g'_{ik}$ factors through a desingularization of $S'_i\times_CS'_k$: 
$$g'_{ik}: X'\tto Y'_{ik}\stackrel{\eps}{\tto} S'_i\times_CS'_k.$$
 By definition of rational Gorenstein singularities, we have $\eps_*\omega_{Y'_{ik}}=\omega_{S'_i\times_CS'_k}$. Since $\omega_{X'/Y'_{ik}}$ is effective, we obtain an inclusion
$$
p_i^*(\omega_{S'_i/C}\otimes {h'_{ik}}^*P_{\xi_k})\otimes p_k^*(\omega_{S'_k/C}\otimes {h'_{ki}}^*P_{\xi_i})\subset
{g'_{ik*}}(\omega_{X'/C}\otimes {f'_i}^*P_{\xi_i}\otimes
{f'_k}^*P_{\xi_k}) 
$$
of sheaves on $S'_i\times_CS'_k $. 
  Pushing forward these sheaves   to
$C$, we obtain
\begin{eqnarray}\label{nn}
h_{i*}(\omega_{S'_i/C}\otimes
{h'_{ik}}^*P_{\xi_k})\otimes h_{k*}(\omega_{S'_k/C}\otimes
{h'_{ki}}^*P_{\xi_i})
&\subset& f'_*(\omega_{X'/C}\otimes
{f'_i}^*P_{\xi_i}\otimes {f'_k}^*P_{\xi_k}) \nonumber\\
&\subset& h^*f_{j*}(\omega_X\otimes
f_i^*P_{\xi_i}\otimes f_k^*P_{\xi_k}),
\end{eqnarray}
where the second
inclusion comes from (\ref{4}). Let $\{\alpha,\beta\}=\{i,k\}$.   Both sheaves $h_{\alpha*}(\omega_{S'_ \alpha/C}\otimes
{h'_{\alpha \beta}}^*P_{\xi_\beta})$   are nef (\cite{V}, Corollary 3.6). On the other hand, for
$\xi_i$ and $\xi_k$ general and torsion, the sheaf in (\ref{nn}) has degree 0 by Proposition \ref{l46}, hence
$$
\deg(h_{\alpha*}(\omega_{S'_ \alpha/C}\otimes
{h'_{\alpha \beta}}^*P_{\xi_\beta}))=0.
$$

By \cite{K1}, Corollary 10.15,
 both sheaves
$R^1h_{\alpha*}(\omega_{S'_ \alpha/C}\otimes
{h'_{\alpha \beta}}^*P_{\xi_\beta})$  are torsion-free
 and generically 0, hence 0.\footnote{Note that up to this point, the proof works in the more general situation where $a_X$ is surjective and $\Alb(X)$ has a 1-dimensional simple  factor $E_1$.}
 On the other hand, we have
$ 
R^1h_{\alpha*}\omega_{S'_\alpha/C} =\cO_C 
$ 
(\cite{ko1}, Proposition 7.6) hence, by the Grothendieck-Riemann-Roch theorem,  
$$
[\ch(h_{\alpha*}(\omega_{S'_\alpha/C}))-\ch(\cO_C)]\Td(C)=\ch(h_{\alpha*}(\omega_{S'_ \alpha/C}\otimes
{h'_{\alpha \beta}}^*P_{\xi_\beta}))\Td(C)
$$
 in the ring of cycles modulo numerical equivalence on $C$.
This implies $\deg(h_{\alpha*}\omega_{S'_\alpha/C}) =0$, and   $h_\alpha$ is locally trivial (see  \eg\
 \cite{BPV}, Theorem III.17.3). Hence $h_{\alpha j}$ is isotrivial (for each $\alpha\in\{i,k\}$).
 \end{proof}

\noindent{\bf Step 2.}
{\em There exist   finite groups $G_i$ acting on $C_{ij}$   such that $C_{ij}/G_i\isom E_j$ and $C_{ik}/G_i\isom E_k$, the surface $S_i$ is birational to the quotient $(C_{ij}\times C_{ik})/G_i$ for the diagonal action of $G_i$, and $h_{ij}$ and $h_{ik}$ are identified with the two projections. }

This is a consequence of the following (probably classical)  result.

\begin{lemm}\label{l53}
Let $S$ be a smooth projective surface with an isotrivial  fibration  $h_1:S\tto \Gamma_1$ onto an irrational curve  with   (constant) irrational  general fiber  $F_1$. 

{\rm a)} There exist a smooth curve $F_2$ and a finite group $H$ acting faithfully on $F_1$ and $F_2$ such that $\Gamma_1$ is isomorphic to $F_2/H$, the surface  $S$ is birationally isomorphic to the diagonal quotient $(F_1\times F_2)/H$, and $h_1$ is the composition $S\isomdra (F_1\times F_2)/H \tto F_2/H\isom \Gamma_1$. Let $h_2$ be the composition $S\isomdra (F_1\times F_2)/H \tto F_1/H $.

{\rm b)} Assume $S$ is of general type. Any isotrivial   fibration $h:S\tto \Gamma$ onto  an irrational curve  $\Gamma$ is either $h_1$ or $h_2$ followed by an isomorphism between $F_1/H$ or $ F_2/H$ with $\Gamma$. 
\end{lemm}

\begin{proof} Item a) is well-known and can be found in
  \cite{Ser}. Let us prove b). Since $\Gamma$ is irrational, $h$ induces an isotrivial fibration $h':(F_1\times F_2)/H \tto \Gamma$. Let $D_2$ be a general (constant irrational) fiber of $h'$.
  The quotient map $\pi: F_1\times F_2\tto (F_1\times F_2)/H$ is \'{e}tale outside a finite set. Hence the Stein factorization $g$ of $h'\circ \pi$ in the diagram
$$
\xymatrix{F_1\times F_2\ar@{->>}[dr]_g\ar@{->>}[r]^-{\pi} &(F_1\times F_2)/H\ar@{->>}[r]^-{h'} & \Gamma\\
&D_1,\ar@{->>}[ur]}
$$
 is also isotrivial, with general fiber $D_2'$   a (fixed) \'{e}tale cover of $D_2$. By a), there is a   base change $D_1'\tto D_1$ and a surjective morphism $t=(t_1, t_2): D_1'\times D_2'\tto F_1\times F_2$. Since $S$ is of general type, $F_1$ and $F_2$ are each of genus $\ge 2$, hence each
 $t_i$ must factor through one of the projections $p_j:D_1'\times D_2'\tto D'_j$.

  If   $h$  factors through  neither $h_1$ nor $h_2$, the curve  $D_2'$ dominates both $F_1$ and $F_2$, hence $t_1$ and $t_2$ cannot factor through $p_1$. Thus they must   factor through $p_2$, which  contradicts the fact that $t$ is surjective.
\end{proof}

Let now   $Y_j$ be a resolution of singularities of the irreducible threefold $S_i\times_{E_j} S_k$ and let 
$Y$ be a resolution of singularities of   the   component of   $Y_1\times_{E_2\times E_3 }S_1$ that dominates both $Y_1$ and $Y_3$.
 After modification of $X$, we obtain a diagram
 \begin{equation}\label{bigd}
\xymatrix@R=10pt{
&X\ar@{->>}[rr]^{g}&&Y\ar@{->>}[dr] \ar@{->>}[dl]\\
&&Y_1\ar@{->>}[dl]_{g_{13}}\ar@{->>}[dr]^{g_{12}}&&Y_3\ar@{->>}[dl]_{g_{32}}\ar@{->>}[dr]^{g_{31}}\\
&S_3\ar@{->>}[dr]_{h_{31}}\ar@{->>}[dl]^{h_{32}}&\square&S_2\ar@{->>}[dl]^{h_{21}} \ar@{->>}[dr]_{h_{23}}& \square&S_1\ar@{->>}[dr]_{h_{13}} \ar@{->>}[dl]^{h_{12}}\\
E_2&&E_1 &&
E_3&&E_2, }
\end{equation}
where the squares are birationally cartesian and the isotrivial  morphisms  $h_{ij}:S_i\tto E_j $ fit  into  diagrams
 \begin{equation*}
\xymatrix
{
&C_{ik}\times C_{ij}\ar[r]^-{p_2}\ar@{-->}[d]^{/G_i}&C_{ij}\ar@{->>}[d]^{/G_i}_{\rho_{ij}}\\
C_{ik}\ar@{^{(}->}[r]^{\rm fiber}&S_i\ar@{->>}[r]_{h_{ij}}&E_j.}
\end{equation*}

\medskip
\noindent{\bf Step 3.}
{\em The threefold $Y$ is dominated by a product of three curves. }

The   dominant maps  $C_{31}\times C_{32}\dra S_3$ and $C_{21}\times C_{23}\dra S_2$ induce a factorization
 $$((\rho_{31}, \rho_{21}), \rho_{32}, \rho_{23}): (C_{31} \times_{E_1}C_{21} ) \times C_{32} \times C_{23}\dra S_3\times_{E_1} S_2\tto E_1\times E_2\times E_3.$$
The (Stein factorization of the)   morphism $Y_1\tto E_2\times E_3$ is therefore isotrivial (its fibers are dominated by   the curve $C_{31} \times_{E_1}C_{21} $).
Thus, $Y$  is dominated by the product 
 \begin{equation}\label{stu}
 (C_{31} \times_{E_1}C_{21})\times(C_{12} \times_{E_2}C_{32})\times (C_{23} \times_{E_3}C_{13}) 
  \end{equation}
 of three (possibily reducible) curves.
 
 \medskip
 
Going back to the proof of Theorem \ref{3cur}, after passing to an \'{e}tale cover,  we may and will assume, from now on,   that the following   holds (\S\ref{23}):
\begin{equation}\label{propp}
\hbox{\em All the irreducible components of $V_0(\omega_X, a_X)$ pass through $0$.}
 \end{equation}
One   checks, using \S\ref{21}, \S\ref{23}, and Proposition \ref{surj}.d), that   if 
 we want $X$ to be birationally covered by a product $\Gamma_1\times \Gamma_2\times \Gamma_3$, with morphisms $ \Gamma_i\tto E_i$, as in the conclusion of the theorem, we must
   have the following.

\medskip
\noindent{\bf Step 4.}
{\em We have
$$V_0(\omega_X, a_X)=(\PE_1\times \PE_2\times \{0\})\cup
 (\PE_1\times
\{0\}\times \PE_3)\cup( \{0\}\times \PE_2\times \PE_3).$$ }
  We already know that $V_0(\omega_X, a_X)$ contains the right-hand-side  (Proposition \ref{l35}.b)) and we must   prove that it has no other components.

\begin{proof}
 Assume $V_0(\omega_X, a_X)$ has  another component $\PT$. It has dimension 2 (Corollary \ref{two}.a)) and, after possibly permuting the indices,  we may assume the neutral component $\PE'_1$ of $\PT \cap (\PE_1\times\PE_2\times \{0\})$ is neither $\PE_1\times\{0\}\times\{0\}$ nor $\{0\}\times \PE_2\times\{0\}$.
 This yields an elliptic curve $E'_1$ which is a quotient of $E_1\times E_2$ which does not factor through either projection. As we saw right before Proposition \ref{l46},  the induced map $f_4: X\tto E'_1$ is a fibration. It factors as
 $$f_4: X  \lraa  S_3\stackrel{h_{34}}{\lraa} E'_1
 $$
 where, by Step 1, $h_{34}$ is isotrivial. By Lemma \ref{l53}.b), $h_{34}$ must factor through one of the projections $h_{31}:S_3\tto E_1$ or $h_{32}:S_3\tto E_2$ so we reach a contradiction.
  \end{proof}

\medskip
\noindent{\bf Step 5.}
{\em The morphism  $g:X\tto Y$ is birational. }

\begin{proof}
Consider, in the diagram (\ref{bigd}),  the   generically finite morphism  $v_1: X\tto Y_1$ and the three fibrations $f'_{\alpha }:Y_1\tto E_\alpha $, for $\alpha\in\{1,2,3\}$.
  Since $X$, $Y_1$, and $S_3$ are all of maximal Albanese dimensions, $\omega_{X/Y_1}$ and $\omega_{Y_1/S_3}$ are effective, hence
 $$h^0(X, \omega_X\otimes a_X^*P_\xi )\ge  h^0(Y_1, \omega_{Y_1}\otimes (f'_1,f'_2)^*P_\xi)\ge  h^0(S_3, \omega_{S_3}\otimes (h_{31},h_{32})^*P_\xi)$$
  for all $\xi \in \PE_1\times \PE_2$.
 Moreover, for $\xi  $  non-zero, we have by Proposition \ref{surj}.d)
  $$h^2(X, \omega_{X}\otimes a_X^*P_\xi)=h^3(X, \omega_{X}\otimes a_X^*P_\xi)=0$$
  hence, since $\chi(X,\cO_X)=0$,
 $$h^0(X, \omega_{X}\otimes a_X^*P_\xi)=h^1(X, \omega_{X}\otimes a_X^*P_\xi).
 $$
 Finally, for $\xi  $ general in $\PE_1\times \PE_2$, we have, as in the proof of Theorem \ref{q1}, since $g_3:X\tto S_3$ has connected fibers,
 $$h^1(X, \omega_X\otimes a_X^*P_\xi )=  h^0(S_3, \omega_{S_3}\otimes (h_{31},h_{32})^*P_\xi).$$
 Therefore, for $\xi\in \PE_1\times \PE_2$ general, we obtain
 $$h^0(X, \omega_X\otimes a_X^*P_\xi)=h^0(Y_1, \omega_{Y_1}\otimes (f'_1,f'_2)^*P_\xi).$$
The induced morphism $ Y\tto E_1\times E_2\times E_3$ is the Albanese mapping of $Y$.
Since in any event, we always have 
$$h^0(X, \omega_X\otimes a_X^*P_\xi)\ge h^0(Y , \omega_Y\otimes a_Y^*P_\xi)\ge h^0(Y_1 , \omega_{Y_1}\otimes (f'_1,f'_2)^*P_\xi) $$
for all $\xi\in  \PE_1\times \PE_2 $, we obtain
\begin{equation}\label{fin}
 h^0(X, \omega_X\otimes a_X^*P_\xi)= h^0(Y , \omega_Y\otimes a_Y^*P_\xi)
\end{equation}
for $\xi  $ general in $\PE_1\times \PE_2$, hence also, by   Step 3, for  
$\xi  $ general in $V_0(\omega_X, a_X)$.
But for  $\xi\notin V_0(\omega_X, a_X)$,   both sides  of (\ref{fin}) vanish.   By Lemma \ref{3-fm} below, we    conclude that $g$ is a birational morphism. 
  \end{proof}

The following lemma (used in the proof above)  is in the spirit of   \cite{hp}, Theorem 3.1.

\begin{lemm}\label{3-fm}Let $X\stackrel{g}{\tto} Y\stackrel{f}{\tto}  A$ be generically finite morphisms between smooth projective threefolds, where $A$ is an abelian threefold,  such that $  f$ and $f\circ g$ are  both minimal. Assume that $X$ is of general type with $\chi(X, \omega_X)=0$ and that there exists an open subset   $U\subset \PA$ with $\codim_{\PA}(\PA\moins U)\ge 2$ such that
 $$h^0(X, \omega_X\otimes g^*f^*P_\xi)=h^0(Y, \omega_Y\otimes f^*P_\xi) $$ for all $\xi\in U $. Then  $g$ is birational.
\end{lemm}

\begin{proof}By \S\ref{23}, we can write $g_*\omega_X\isom   \omega_Y\oplus \cE$, and we need to show that the sheaf $\cE$ is zero. Since $\cE$ is torsion-free and $f$ is generically finite, it is  sufficient to prove $f_* \cE =0$.

As we saw at the beginning of \S\ref{s3}, we have $q(X)=3$, hence $f\circ g$ 
  is the Albanese mapping of $X$. By Proposition \ref{surj}.d), for each $i\in\{2,3\}$, we then  have $\{0\}=V_i(\omega_X, f\circ g)=V_i(g_*\omega_X, f)$, hence  $V_i(f_* \cE)\subset \{0\}$.  
  
 Since  $q(X)=3$, we also have $q(Y)=3$, hence $h^i(Y, g_*\omega_X)=h^i(Y, \omega_Y)$.  It follows that
 $V_i(f_* \cE)$ is empty.

The assumption $\chi(X, \omega_X)=0$ implies   $\chi(Y, \omega_Y)=0$ by (\ref{chi}). Thus,
 $$V_0(f_* \cE)=V_1(f_* \cE)\subset \PA\moins  U.$$
Since $\codim_{\PA}(\PA\moins U)>1$, the sheaf $f_* \cE $ is therefore M-regular in the sense of \cite{pp1}, Definition 2.1 (see also Remark 2.3), hence continuously \gge\ (\cite{pp1}, Definition 2.10 and Proposition 2.13). Since $H^0(A,f_* \cE\otimes P_\xi)=0$ for all $\xi\in U$,  we obtain $f_* \cE =0$.
\end{proof}

 Let us summarize what we know.
Let $\{i, j, k\}=\{1,2,3\}$.
The curve  $C_{ij}$ is the  (constant) general fiber of the isotrivial fibration $S_i\tto E_k$; it is acted on by a group $G_i$ and $C_{ij}/G_i\isom E_j$ (Step 2).
The fibration $g_i: X\tto S_i$ is also isotrivial; as we saw in Step 3, its general fiber $C_i$ is dominated by the curve 
$C_{ji}\times_{E_i}C_{ki}$ but also maps onto $C_{ji}$ and $C_{ki}$. Finally,  a general fiber $F_k$ of the isotrivial fibration $f_k: X\tto E_k$ is an isotrivial fibration over   $C_{ij}$  with (constant)   general fiber  $C_i$. The situation is summarized in the following diagram:
$$\xymatrix@C=20pt@R=20pt
{
& C_{ij}\ar@{^{(}->}[rr]^-{\rm fiber} &&S_i\ar@{->>}[dr]^{h_{ik}}\\
F_k\ar@{^{(}->}[rr]^-{\rm fiber}\ar@{->>}[ur]^-{{\rm fiber}\ C_i}\ar@{->>}[dr]_-{{\rm fiber}\ C_j} & &X\ar@{->>}[ur]^(.4){g_i}\ar@{->>}[rr]^{f_k}\ar@{->>}[dr]_(.4){g_j}&& E_k\\
& C_{ji}\ar@{^{(}->}[rr]^-{\rm fiber} && S_j.\ar@{->>}[ur]_{h_{jk}}
}$$
By Lemma \ref{l53},    there exists a finite group $H_k$ acting faithfully on $C_i$ and $C_j$ such that $C_{ij}\isom C_j/H_k$, $C_{ji}\isom C_i/H_k$, and $F_k$ is isomorphic to the diagonal quotient $(C_i\times C_j)/H_k$. Moreover, the maps to $C_{ij}$ and $C_{ji}$ are the natural projections.
  So we have diagrams
\begin{equation}\label{cu}
\xymatrix@C=30pt@R=8pt
{& C_{ji}\ar@{->>}[dr]^{/G_j}\\
C_i\ar@{->>}[ur]^{/H_k}\ar@{->>}[dr]_{/H_j}&& E_i.\\ 
& C_{ki}\ar@{->>}[ur]_{/G_k}}
\end{equation}
 
Let $D_1$   be the Galois closure of $C_1$ over $E_1$ and set  $G=\Gal(D_1/E_1)$. Let $\{j, k\}=\{2, 3\}$. There is  a normal subgroup  $N_j\lhd G$ such that $G_j=G/N_j$ and $G$ acts on $C_{jk}$ via this quotient. By Step 2, the surface $S_j$ is birationally isomorphic to 
 $(C_{j1}\times C_{jk})/G_j$, hence to  $(D_1\times C_{jk})/G$. Therefore, the   modification 
$Y_1$ of $S_2\times_{E_1}S_3$   (see   (\ref{bigd}))
is birationally isomorphic to
 $(D_1\times C_{23}\times C_{32})/G$.

\medskip
\noindent{\bf Step 6.}
{\em The group $G$ is isomorphic to $ \Z/2\Z$. }

  We begin with a lemma which is probably well-known. We denote by $\Irr(G)$  the set  of  isomorphism classes of   irreducible representations  of $G$.

\begin{lemm}\label{l55}
Let $E$ be an elliptic curve and let $\pi:D\tto E$ be a   Galois cover with group $G$. We can write
 $$\pi_*\cO_D= \bigoplus_{\chi\in \Irr(G)}\bigoplus_{i}\cV_{\chi, i} ,$$
 where each vector bundle 
$\cV_{\chi, i}$ is  semistable and $G$-invariant, and
    the representation of $G$   on the general fiber of each $\cV_{\chi, i}$ is a direct sum of $\chi$.  Moreover, for each $\chi\ne 1$, the dual vector bundle 
$\cV^\vee_{\chi, i}$  is either ample or a direct sum of non-zero torsion line bundles.
\end{lemm}

\begin{proof}
The groups $G$ acts on $\pi_*\cO_D$. Identifying each representation $\chi$ with its character, we consider the endomorphism  
$
\sum_{g\in G}\chi(g)g$ of $\pi_*\cO_D$ and we denote by $\cV^\vee_\chi$ its image. We then have  
\begin{equation}\label{piod}
\pi_*\cO_D=\bigoplus_{\chi} \cV_\chi,
\end{equation}
where the general fiber of  $\cV_\chi$ is a (nonzero) direct sum of $\chi$ as a $G$-module.

The Harder-Narasimhan filtration 
$$
0=\cV^\ell_\chi\subset \cV_\chi^{\ell-1}\subset  \cdots \subset \cV_\chi^0=\cV_\chi
$$
  is preserved by the $G$-action. Serre duality shows that since we are on an elliptic curve, all the corresponding extensions are trivial, hence $\cV_\chi$ is the direct sum of the $G$-invariant semistable bundles $\cV_{\chi, i}=\cV_\chi^i/\cV_\chi^{i+1}$, for $0\le i< \ell$.

As a direct summand of $\pi_*\omega_D=(\pi_*\cO_D)^\vee$, each vector bundle $\cV^\vee_{\chi, i}$ is   nef (\cite{V}, Corollary 3.6). Moreover, it is ample if it has positive degree.
Consider the maximal degree-$0$ subsheaf $\cF$ of $\pi_*\cO_D$, \ie\  the  direct sum of all $\cV_{\chi,i}$ that have degree 0.  By \cite{KP},  Lemma 3.2 and 3.4,   $\cF$ is a $G$-invariant subalgebra and induces an \'{e}tale cover of $E $, hence is a direct sum of torsion line bundles. 
 \end{proof}
 
 Let us continue with the Galois cover $\pi:D\tto E$  with group $G$ as in the lemma and assume moreover that  for each $j\in\{2,3\}$, we have a Galois cover $\pi_j: D_j\tto  E_j$   with Galois group $G_j=G/N_j$, where $g(D_j)\geq 2$ and $E_j$ is an elliptic curve.

 Then $G$ acts on $D_1\times D_2\times D_3$ diagonally. Let $Z$ be the quotient, let $\eps:Y\tto Z$ be  a resolution, and consider
 $$t: Y\stackrel{\varepsilon}{\tto} Z \tto E_1\times E_2\times E_3.$$

 \begin{lemm}\label{2elements}Assume 
$$
V_0(\omega_Y, t)\subset (\PE_1\times \PE_2\times \{0\})\cup
 (\PE_1\times
\{0\}\times \PE_3)\cup( \{0\}\times \PE_2\times \PE_3)
$$
 and $V_2(\omega_Y, t)\cup V_3(\omega_Y, t)\subset \{0\}$. Then  $N_2=N_3$ and $G_2\isom G_3\isom \Z/2\Z$.
\end{lemm}

\begin{proof}We decompose $\pi_*\cO_D$ as in (\ref{piod}) and write similarly
$$\pi_{j*}\cO_{D_j}= \bigoplus_{\mu\in \Irr(G_j)}\bigoplus_i\cV_{\mu, i}^j.$$
 Since quotient singularities are rational, we have as in Example \ref{eel}
$$
t_*\omega_Y\isom  (q_*\cO_Z)^\vee \isom  \big((\pi_*\cO_D)^ \vee\boxtimes ((\pi_{2*}\cO_{D_2})^ \vee\boxtimes ((\pi_{3*}\cO_{D_3})^ \vee\big)^G.
$$
Let $\mu$ be an non-trivial element of $\Irr(G_2)$.  Since $G_2$ is a quotient of $G$, the representation $\mu$ and its complex conjugate $\overline{\mu}$ are also in $\Irr(G)$. Then, the vector bundle
 $$ \cG:=(\cV^\vee_{\overline{\mu},1}\boxtimes \cV^{2\vee}_{\mu,1}\boxtimes\cO_{E_3})^G$$ 
on $E_1\times E_2\times E_3 $  is a non-zero direct summand of both $\cV^\vee_{\overline{\mu},1}\boxtimes \cV^{2\vee}_{\mu,1}\boxtimes\cO_{E_3}$ and $t_*\omega_Y$.
 
 Assume that  $\cV^{2}_{\mu,1}$ has degree 0,  hence is a direct sum of non-trivial torsion line bundles. 
\begin{enumerate}
\item[$\bullet$]   If $\deg(\cV^\vee_{\overline{\mu},1})=0$, the sheaf $\cG$ is a direct sum of non-trivial torsion line bundles, which is impossible since $V_3(\cG)\subset V_3(\omega_Y, t)=\{0\}$.
\item[$\bullet$]   If $\cV^\vee_{\overline{\mu},1}$ is ample,   we can write 
  $$\cG=\bigoplus_k (\cG_k\boxtimes P_{\xi_k}\boxtimes\cO_{E_3}),$$
   where $\cG_k$ is a direct summand of $\cV^\vee_{\overline{\mu},1}$, hence ample, and the $\xi_k$ are  non-zero  torsion points  in $\PE_2$. This is again impossible, because $V_2(\cG)\subset V_2(\omega_Y, t)=\{0\}$.
\end{enumerate}

Therefore,  $\cV^{j\vee}_{\mu,1}$ is ample for all $\mu$ non-trivial in $\Irr(G_j)$.  

If $\Card(G_2)>2$, or if $N_2\ne N_3$, we may take non-trivial  $\chi\in\Irr(G)$, $\mu\in\Irr(G_2)$, and $\nu\in\Irr(G_3)$ such that $\chi$ is a subrepresentation of $\mu\otimes \nu$. The vector bundle 
$$\cH:=(\cV^\vee_{\overline{\chi}, 1}\boxtimes \cV^{2\vee}_{\mu, 1} \boxtimes \cV^{3\vee}_{\nu, 1})^G$$ 
is then non-zero  and a direct summand of  $t_*\omega_Y$ (and $\cV^{2\vee}_{\mu, 1} $ and $ \cV^{3\vee}_{\nu, 1}$ are ample). 

If $\cV^\vee_{\overline{\chi}, 1}$ is ample, since $\cH$ is a direct summand of $\cV^\vee_{\overline{\chi}, 1}\boxtimes \cV^{2\vee}_{\mu, 1} \boxtimes \cV^{3\vee}_{\nu, 1}$, we have $V_m(\cH)=\vide$ for all $m\in\{1,2,3\}$. Hence $h^0(E_1\times E_2\times E_3, \cH\otimes P_\xi)$ is a non-zero constant for all $\xi\in \PE_1\times \PE_2\times \PE_3$ and $V_0(\cH)=\PE_1\times \PE_2\times \PE_3$, which contradicts our assumptions.

 If $\cV^\vee_{\overline{\chi}, 1}$ is a direct sum of non-trivial torsion line bundles, we may write 
 $$\cH=\bigoplus_k (P_{\xi_k}\boxtimes \cH_k),$$
  where the  $\xi_k$ are non-zero   torsion points in $\PE_1$ and $\cH_k$ is a direct summand of $\cV^{2\vee}_{\mu, 1} \boxtimes \cV^{3\vee}_{\nu, 1}$. Then $V_0(\cH)$, hence also $V_0(\omega_Y, t)$, contains   $\{-\xi_1\}\times\PE_2\times\PE_3$, which contradicts our assumptions.
\end{proof}

We now apply this second lemma to the Galois covers $\pi:D_1\tto E_1$, $\pi_2:C_{23}\tto E_2$, and $\pi_3:C_{32}\tto E_3$. The variety $Y$ of the lemma is the variety $Y_1$ of the proof, and since 
$V_0(\omega_{Y_1}, t)\subset V_0(\omega_X, a_X)$ (see \S\ref{23}), the hypotheses of the lemma are satisfied (Step 4).

We obtain $N_2=N_3$, hence the coverings $C_{ji}\tto E_i$ and $C_{ki}\tto E_i$ are the same (see (\ref{cu})), and also $G/N_j  \isom \Z/2\Z$, so that they are double covers. Denote them by $C'_i\tto E_i$. 
By the proof of Step 3 (see (\ref{stu})),  $X$ is birational to $ (C'_1\times C'_2\times C'_3)/( \Z/2\Z)$. Since the latter variety contains no rational curves, there is a birational {\em morphism} from $X$ to it. This finishes the proof of Theorem \ref{3cur}.
\end{proof}

 \section{Varieties with $P_1=1$}\label{p11}

It follows from \cite{ueno} and \ref{215}  that varieties $X$ of maximal Albanese dimension and $P_1(X)=1$ satisfy $\chi(X,\omega_X)=0$. We presented in Example \ref{chh} a construction of Chen and Hacon of such a variety which is in addition of general type. We gather here some properties of these varieties (most of them taken from \cite{ueno}).

\begin{prop}
Let $X$ be a smooth projective variety of maximal Albanese dimension $n$, with  $P_1(X)=1$.

\begin{enumerate}
\item[a)] We have an isomorphism
$$a_X^*:\bigwedge^\bullet H^0(A,\Omega_A)\isom H^0(X,\Omega^\bullet_X).$$
In particular, $h^j(X,\cO_X)=\binom{n}{j}$ for all $j$, hence $\chi(X,\omega_X)=0$, and 
the  Albanese mapping $a_X:X\tto \Alb(X)$   is surjective.
\item[b)]  The point 0 is isolated in $V_0(\omega_X,a_X)$.
\end{enumerate}
\end{prop}

\begin{proof}
 Replacing $X$ with a modification, we may assume that there is a factorization $a_X:X  \tto   Z\to \Alb(X)$, where   $  Z$ is a desingularization of $ a_X(X)$,
so that $ P_1(  Z)\le P_1(X)=1$.   It follows from  \cite{ueno} (or   \cite{mor}, Corollary (3.5)) that $a_X(X)$ is a translate of an abelian subvariety of $\Alb(X)$, hence $a_X$ is surjective. 
Item a) then follows from another result of Ueno (\cite{ueno}, or   \cite{mor}, Corollary (3.4)).

By \S\ref{23}, we can write  $a_{X*}\omega_X \isom \omega_A \oplus \cE\isom  \cO_A \oplus \cE $. The sheaf $\cE$ then satisfies   $V_i(\cE)\moins\{0\}=V_i(\omega_X,a_X)\moins\{0\}$ for all $i$.
Since $1=P_1(X)=1+h^0(A, \cE)$, the point $0$ is not in the closed set $ V_0(\cE)$, hence is isolated in $V_0(\omega_X,a_X)$. This proves b).
\end{proof}

\begin{rema}\upshape
Regarding item b), to be more precise, a smooth projective variety $X$ of maximal Albanese dimension satisfies $P_1(X)=1$ {\em if and only if} 0 is isolated in $V_0(\omega_X, a_X)$.
\end{rema}

\begin{theo}\label{th63}
Let $X$ be a smooth projective threefold  of maximal Albanese dimension and of general type.
If $P_1(X)=1$, the variety $X$  is  a modification of an abelian \'{e}tale cover of a Chen-Hacon threefold.
\end{theo}

 It is then very easy to describe all smooth projective threefolds $X$  of maximal Albanese dimension and of general type, with $P_1(X)=1$. Start from a Chen-Hacon threefold $Y$ as in Example \ref{chh}, with Albanese mapping $a_Y:Y\tto E_1\times E_2\times E_3$. It satisfies
 $$V_0(\omega_Y, a_Y)=\{0\}\cup (\PE_1\times \PE_2\times\{\xi_3\})\cup(\PE_1\times\{\xi_2\}\times \PE_3)\cup(\{\xi_1\}\times\PE_2\times \PE_3),
$$
where each $\xi_j\in\PE_j$ has order 2. Take an isogeny $A\tto E_1\times E_2\times E_3$   corresponding to a (finite) subgroup of $\PE_1\times\PE_2\times \PE_3$ which contains none of the points $\xi_1,\xi_2,\xi_3$. Finally, take   for $X$ a modification of   $Y\times_{E_1\times E_2\times E_3}A$.

\begin{proof}[Proof of the theorem]Replacing $a_X$ with it Stein factorization, we will assume that $X$ is normal and $a_X$ is finite. By Theorem \ref{3cur}, there exist elliptic curves $E_1$, $E_2$, and $E_3$, double coverings $\rho_i: C_i\tto E_i$,  with involution  $\iota_i$, and a commutative diagram
$$
\xymatrix@C=30pt
 {C_1\times C_2\times C_3  \ar@{->>}[dr]^{(\rho_1,\rho_2,\rho_3)}\ar@{->>}[d]\\
\widetilde X\ar@{->>}[d]\ar@{->>}[r]^-{a_{\widetilde X}}\ar@{}[dr]|\square & E_1\times E_2\times E_3\ar@{->>}[d]^\eta\\
X\ar@{->>}[r]_{a_X} &\Alb(X),}
$$
where $\eta$ is an isogeny (the variety $\widetilde X$ is the variety $Z$ of Example \ref{eel}) and both $a_X$ and $a_{\widetilde X}$ are $( \Z/2\Z)^2 $-Galois coverings. In particular, $X$ has rational singularities. 

We denote by $K$ the (finite) kernel of $\eta$ and by $K_i$ the image of the projection $K\to E_i$, so that $K$  is a subgroup of $ \widetilde K:=K_1\times K_2\times K_3$. 
The elliptic curve $F_i:=E_i/K_i$ embeds in $\Alb(X)$; let $\pi_i: \Alb(X)\tto A_i$ be the quotient. The natural
morphism $h_i: X\tto A_i$ is an isotrivial fibration and we denote by $D_i$ its general (constant) fiber. Now we consider the square restricted to fibers:
\begin{eqnarray*}
\xymatrix@C=40pt{
C_i\ar@{->>}[r]_{\rho_i}^{2:1} \ar@{->>}[d]_-{\rm{\acute etale}} & E_i\ar@{->>}[d]_-{\lambda_i}^-{\rm{\acute etale}}\\
D_i\ar@{->>}[r]_{t_i}^{4:1} & F_i,
}
\end{eqnarray*}
where $t_i$ is a $(\Z/2\Z)^2$-cover. Since $D_i\times_{F_i}E_i$ is disconnected and factors as $C_i\sqcup C_i\tto C_i \stackrel{\rho_i}{\tto}E_i$, there is a non-zero 2-torsion point 
$\xi_i\in \PF_i$ such that $\xi_i\in \Ker(t_i^*)\cap\Ker(\lambda_i^*)$, the morphism  $t_i$ factors as $D_i\stackrel{s_i}{\tto}D_i'\tto F_i$, where $s_i$ is a double \'{e}tale cover, and $C_i\isom D_i'\times_{F_i}E_i$. 
It follows that  the group $K_i$ acts on  $C_i$, and the involution $\iota_i$ and the $K_i$-action  commute. 
Therefore, $\widetilde K$ acts on $C_1\times C_2\times C_3$ and this  action commutes with the involution $(\iota_1, \iota_2, \iota_3)$. It follows that $\widetilde K$ acts on $\widetilde{X}$ and the Albanese mapping $a_{\widetilde X}$ is $\widetilde K$-equivariant. Set $Y:=\widetilde{X}/\widetilde K$.

 \begin{lemm}The quotient morphism $\widetilde{X}\tto Y $ is \'{e}tale and factors as $\widetilde{X}\tto X\tto Y$. The variety  $Y$ has maximal Albanese dimension, is of general type,  $P_1(Y)=1$, and its Albanese variety  is $F_1\times F_2\times F_3$.
 \end{lemm}
 
 \begin{proof}We have a   cartesian diagram
\begin{eqnarray}\label{d1}
\xymatrix@C=50pt
{
 \widetilde{X}\ar@{}[dr]|\square\ar@{->>}[d]^{/\widetilde K}\ar@{->>}[r]^-{a_{\widetilde X}}   & E_1\times E_2\times E_3\ar@{->>}[d]^{/\widetilde K}_\lambda\\
 Y\ar@{->>}[r]^-{a_Y} & F_1\times F_2\times F_3.
 }
 \end{eqnarray}
  Since the rightmost quotient morphism is  \'{e}tale, so is  the leftmost quotient morphism. Since the top morphism is the Albanese mapping of $\widetilde X$, the bottom morphism  is the Albanese mapping of $Y$.
 Furthermore, since  $K$ is a subgroup of $\widetilde K$, the leftmost quotient morphism factors as $\widetilde{X}\tto X\tto Y$. Therefore, $P_1(Y)=1$.
 \end{proof}

We now claim that $Y$ is a Chen-Hacon threefold.  By Theorem \ref{3cur} (or the diagram (\ref{d1})),   $a_Y$ is a $(\Z/2\Z)^2$-Galois covering and by \cite{P}, we can write
$$a_{Y*}\cO_Y=\bigoplus_{\chi\in ( \Z/2\Z)^{2*}}L_{\chi}^\vee=\cO_{\Alb(Y)}\oplus L_{\chi_1}^\vee\oplus L_{\chi_2}^\vee\oplus L_{\chi_3}^\vee,$$
or equivalently, since $Y$ has rational singularities, 
\begin{equation}\label{e15}
a_{Y*}\omega_Y=\cO_{F_1\times F_2\times F_3}\oplus L_{\chi_1}\oplus L_{\chi_2}\oplus L_{\chi_3},
\end{equation}
 where $L_{\chi_1}$, $ L_{\chi_2}$, and $ L_{\chi_3}$ are line bundles on 
$F_1\times F_2\times F_3$. Moreover, by  \cite{P}, Theorem 2.1, we have the following ``building data'': there are effective divisors $D_1$, $D_2$, and $D_3$ on $F_1\times F_2\times F_3$
satisfying:
$$L_{\chi_i}+L_{\chi_j}\siml L_{\chi_k}+D_k\quad{\text{and}}\quad L_{\chi_i}^2\siml D_j+D_k$$
for any $\{i, j, k\}=\{1, 2, 3\}$. These data pull back to the analogous building data on $\widetilde X$, hence
$\lambda^*D_i$ is the pull-back on $E_1\times E_2\times E_3$ of the branch divisor $\Delta_i\siml 2\delta_i$ of $\rho_i$. It follows that there exists an ample line bundle $\delta'_i$ on $F_i$ which pulls back to $\delta_i$ on $E_i$ and such that $D_i$ is also the pull-back on $F_1\times F_2\times F_3$ of a divisor  $\Delta'_i\siml 2\delta'_i$ on $F_i$. Let $L'_i$ be the pull-back  on $F_1\times F_2\times F_3$ of $\delta'_i$. Because of the relations $L_{\chi_i}^2\siml D_j+D_k$, we can write
$$L_{\chi_i}\isom P_{\xi_i}\otimes (L'_j\otimes P_{\xi_{i,j}})\otimes (L'_k\otimes P_{\xi_{i,k}}),$$
where $\xi_i\in \PE_i$,  $\xi_{i,j}\in \PE_j$, and $\xi_{i,k}\in \PE_k$ are 2-torsion points. From
 (\ref{e15}) and the fact that $P_1(Y)=1$, we deduce $H^0(E_i,L_{\chi_i})=0$, hence 
each $\xi_i$ has order 2 and is in the kernel of $\widehat \lambda_i$. From the relations $L_{\chi_i}+L_{\chi_j}\siml L_{\chi_k}+D_k$, we deduce 
$$\xi_{i,k}+\xi_{j,k}=\xi_k\quad{\text{and}}\quad \xi_{i,j}+\xi_j=\xi_{k,j}.$$
 Since $\lambda_1^*\xi_1=0$, we may always change $L'_1$ to $L'_1\otimes P_{\xi_1}$, so we may assume $\xi_{3,1}=0$ and similarly, $\xi_{1,2}=0$ and $\xi_{2,3}=0$.
The $\cO_{F_1\times F_2\times F_3}$ algebra $a_{Y*}\cO_Y$ is then    the algebra associated to a Chen-Hacon threefold (see (\ref{e5})).
 We conclude that $Y$ is a Chen-Hacon threefold.
\end{proof}

\section{A conjecture}\label{s7}

As mentioned in the introduction, we end this article with a conjecture on the possible general structure of
 smooth projective  varieties $X$ of maximal Albanese dimension, of general type, with $\chi(X, \omega_X)=0$. 

 \begin{conj*}
 Let $X$ be a smooth projective  variety of maximal Albanese dimension, of general type, with $\chi(X, \omega_X)=0$. Then there exist a smooth projective variety $X'$, a morphism $X'\tto X$ which is a composition of modifications and abelian  \'etale covers, and  a   fibration $g:X'\tto Y$ with general fiber $F$, such that $0<\dim(Y)<\dim(X)$ and
 \begin{itemize}
 \item [{\rm a)}] either $g$ is isotrivial;
 \item [{\rm b)}] or $\chi(F, \omega_F)= 0$.
 \end{itemize}
 \end{conj*}

\begin{remas}\label{remm}
1)  Conversely, in the situation   b) above,  $\chi(X, \omega_X)= 0$ (\cite{hp},  Proposition 2.5).
 Moreover,  $\Alb(X)$ has at least 4 simple factors by Corollary \ref{two}.b) and Proposition \ref{surj}.b).
Of course, in case a), without further constraints, one might have $\chi(X, \omega_X)>0$, but we were unable to find necessary and sufficient conditions on the  isotorivial fibration $g$ (assuming $X$ does not fall into case b)) to ensure $\chi(X, \omega_X)=0$.

3) If we are {\em not} in case b), 
it follows from Lemma \ref{le46} that if $X'\tto X$ is any composition of modifications and abelian \'etale covers, we have $q( X')=\dim(X)$ and any morphism from $X'$ to a curve of genus $\ge 2$ is constant.
\end{remas}

The Ein-Lazarsfeld example  (Example \ref{eel}) falls into case a) of the conjecture, and not into case b) by Remark \ref{remm}.1) above. We present an example that falls  into case b), but not into case a). It is basically a non-isotrivial fibration whose general fibers are Ein-Lazarsfeld threefolds.

 \begin{exam}\label{eell}  Consider a smooth projective curve $C$ of genus $\geq 2$, elliptic curves $E_1$, $E_2$, and $E_3$, and smooth   double coverings $ S_j\tto C\times E_j$ ramified along ample divisors. Denote by $\iota_j$ the corresponding involution of $S_j$.
  We may moreover assume that the fibrations $f_j: S_j\tto C$ are all semistable and   not isotrivial.

The fourfold $T:=S_1\times_C S_2\times_C S_3$ has only rational Gorenstein singularities, and so does its quotient $Z$   by the involution $\iota_1\times\iota_2\times\iota_3$. Let $\eps:X\tto Z$ be a desingularization. We have a diagram
$$
\xymatrix@C=15pt
{&&& T\ar@{->>}[dlll]\ar@{->>}[dll]\ar@{->>}[dl]\ar@{->>}[d]_g^{2:1}\\
S_1\ar@{->>}[d]^{2:1}&S_2\ar@{->>}[d]^{2:1}&S_3\ar@{->>}[d]^{2:1} &Z\ar@{->>}[d]_-f^{4:1}&X\ar@{->>}[l]_\eps\\
C\times E_1\ar@{->>}[drrr]&C\times E_2\ar@{->>}[drr]&C\times E_3\ar@{->>}[dr]&C\times E_1\times E_2\times E_3\ar@{->>}[d] \\ 
&&&C.}
$$
The variety $X$ is of general type and    has maximal Albanese dimension because $C\times E_1\times E_2\times E_3$ does. A general fiber of the fibration $ X\tto C$ is one of the examples constructed in Example \ref{eel}, hence  $\chi(X, \omega_X)= 0$ by \cite{hp},  Proposition 2.5, and $X$ falls into case b) of the conjecture. One can prove that  it does not fall into case a).

\end{exam}

\end{document}